\documentclass[11p,reqno]{amsart}
\usepackage{amsmath}
\usepackage{amsfonts}
\usepackage{amssymb}
\usepackage{graphicx}
\usepackage{color}

\newtheorem{theorem}{Theorem}[section]\newtheorem{thm}[theorem]{Theorem}
\newtheorem*{theorem*}{Theorem}
\newtheorem{lemma}{Lemma}[section]

\newtheorem{prop}{Proposition}[section]
\newtheorem{remark}[theorem]{Remark}

\def\Ric{\text{Ric}}

\def\a{\alpha}
\def\l{\lambda}

\def\p{\partial}

\def\R{\Bbb R}

\def\vp{\varphi}

\def\L{{\mathcal L}}

\def\Ric{\operatorname{Ric}}

\newcommand{\eps}{{\varepsilon}}
\numberwithin{equation}{section}

\begin{document}

\title[First nonzero eigenvalue of the weighted $p$-Laplacian]{Sharp lower bound for the first eigenvalue of the weighted $p$-Laplacian II}

\author{Xiaolong Li}
\address{Department of Mathematics, University of California, Irvine, Irvine, CA 92697, USA}
\email{xiaolol1@uci.edu}

\author{Kui Wang}\thanks{The research of the second author is supported by NSFC No.11601359} %\textcolor[rgb]{0.00,0.00,1.00}{}.}
\address{School of Mathematical Sciences, Soochow University, Suzhou, 215006, China}
\email{kuiwang@suda.edu.cn}


\maketitle

\begin{abstract}   
Combined with our previous work \cite{LW19eigenvalue}, we prove sharp lower bound estimates for the first nonzero eigenvalue of the weighted $p$-Laplacian with $1< p< \infty$ on a compact Bakry-\'Emery manifold $(M^n,g,f)$, without boundary or with a convex boundary and Neumann boundary condition, 
satisfying $\Ric+\nabla^2 f \geq \kappa \, g$ for some $\kappa \in \mathbb{R}$. 
\end{abstract}

%\input{intro.tex}
%\input{new.tex}
%\input{sharpness.tex}
%\tableofcontents

\section{Introduction}

The determination of lower bounds for the first nonzero eigenvalue of elliptic operators is an important issue in both mathematics and physics, since this constant determines the convergence rate of numerical schemes in numerical analysis, describes the energy of a particle in the ground state in quantum mechanics, and determines the decay rate of heat flows in thermodynamics. 
Given its physical and mathematical significance, sharp lower bounds of the first nonzero eigenvalue of the Laplace-Beltrami operator on a compact Riemannian manifold
or the $f$-Laplacian on a compact Bakry-\'Emery manifold (without boundary or with a convex boundary and Neumann boundary condition) in terms of geometric data have been established via the efforts of many mathematicians including \cite{BQ00, Kroger92, Li79, LY80, ZY84} by the year 2000. Simple alternative proofs via the estimates of modulus of continuity were found in recent years in \cite{AC13, AN12}. 

In the last two decades, much attention has been focused on eigenvalue problems of nonlinear operators, especially the $p$-Laplacian $\Delta_p$ and the weighted $p$-Laplacian $\Delta_{p,f}$, defined for $1<p<\infty$ by
\begin{equation*}
    \Delta_{p} u:= \text{div} (|\nabla u|^{p-2}\nabla u)
\end{equation*}
and 
\begin{equation*}
    \Delta_{p,f} u:= e^f\text{div} (e^{-f}|\nabla u|^{p-2}\nabla u),
\end{equation*}
respectively. 
On a compact Riemannian manifold satisfying $\Ric \geq (n-1)\kappa g$ for $\kappa \in \R$, sharp lower bounds for the first nonzero eigenvalue of $\Delta_p$  have been obtained in \cite{Matei00} for $\kappa >0$, \cite{Valtorta12} for $\kappa =0$, and \cite{NV14} for $\kappa <0$, as well as a sharpened result for $\kappa >0$.
We refer the reader to \cite[Theorem 1.3]{LW19eigenvalue} for a unified statement in terms of eigenvalue comparisons with associated one-dimensional models. 

The purpose of this paper is to prove sharp lower bounds for the first nonzero eigenvalue of $\Delta_{p,f}$ on compact Bakry-\'Emery manifolds in terms of dimension, diameter, and lower bound of the Bakry-\'Emery Ricci tensor $\Ric+\nabla^2 f$, thus completing our previous work \cite{LW19eigenvalue} and generalizing the above-mentioned results for compact Riemannian manifolds with Ricci curvature lower bounds. 

Recall that a triple $(M^n,g,f)$ consisting of an $n$-dimensional Riemannian manifold $(M^n,g)$ and a function $f\in C^{\infty}(M)$ is called a Bakry-\'Emery manifold if it satisfies 
%be a compact Bakry-\'Emery manifold satisfying 
$$\Ric +\nabla^2 f \geq \kappa\, g$$ 
for some $\kappa \in \R$. 
Here $\Ric$ denotes the Ricci curvature of $(M,g)$ and $\nabla^2 f$ denotes the Hessian of the function $f$. By taking $f$ to be a constant function, we see that Bakry-\'Emery manifolds include all Riemannian manifolds with a lower bound on the Ricci curvature. 
The tensor $\Ric+\nabla^2 f$ is called the Bakry-\'Emery Ricci tensor and it shares many important properties as the Ricci curvature, see for instance \cite{Lott03} and \cite{WW09}.  

The first nonzero eigenvalue of 
$\Delta_{p,f}$ on a closed Bakry-\'Emery manifold $(M^n,g,f)$, denoted by  $\l_{p,f}$, is defined by
\begin{equation*}\label{lambda def}
    \lambda_{p,f} =\inf \left\{\frac{\int_M |\nabla u|^p e^{-f} d\mu}{\int_M |u|^p e^{-f} d\mu} : u \in W^{1,p}(M, e^{-f}d \mu)\setminus\{0\}, \int_M |u|^{p-2} u \, e^{-f} d \mu =0 \right\}.
\end{equation*}
It was shown in \cite{Tolksdorf84} that this infimum is achieved by an eigenfunction $u \in C^{1,\a}(M)$ satisfying the Euler-Lagrange equation 
\begin{equation*}
    \Delta_{p,f} u =- \l_{p,f} |u|^{p-2} u. 
\end{equation*}
In case $\p M \neq \emptyset$, the Neumann boundary condition $\frac{\p u}{\p \nu}=0$ on $\p M$ is always imposed, where $\nu$ is the outward unit normal vector along $\p M$. In our recent work \cite{LW19eigenvalue}, we obtained sharp lower bound estimates for $\l_{p,f}$ in terms of $p$, $\kappa$, and the diameter of $M$, provided that either $1<p\leq2$ or $\kappa \leq 0$. In this paper, we take care of the remaining case $p>2$ and $\kappa >0$. 

Combined with our results in \cite{LW19eigenvalue}, we prove that
\begin{thm}\label{Thm Main}
Let $(M^n,g,f)$ be a compact Bakry-\'Emery manifold (possibly with $C^2$ convex boundary) with diameter $D$ and $\Ric +\nabla^2 f \geq \kappa \, g$ for $\kappa \in \R$.
For $1<p<\infty$, let $\l_{p,f}$ be the first nonzero eigenvalue of the weighted $p$-Laplacian $\Delta_{p,f}$ (with Neumann boundary conditions if $\partial M \neq \emptyset$).
Then we have
\begin{equation}\label{eigenvalue estimate}
    \lambda_{p,f} \geq \mu_p(\kappa,D),
\end{equation}
where $\mu_p(\kappa,D)$ is the first nonzero Neumann eigenvalue of the one-dimensional eigenvalue problem
\begin{equation}\label{ODE}
    (p-1)|\vp'|^{p-2}\vp'' - \kappa\, t \, |\vp'|^{p-2}\vp' =-\l |\vp|^{p-2}\vp
\end{equation}
on $[-D/2,D/2]$. 
\end{thm}
When $\kappa =0$, the ODE \eqref{ODE} can be solved explicitly (see \cite{DP05} or \cite{Valtorta12}) and we have in particular that 
\begin{equation*}
    \mu_{p}(0, D) = (p-1)\left(\frac{\pi_p}{D} \right)^p, \text{ where } \pi_p =\frac{2\pi}{ p \sin (\pi/p)}.
\end{equation*}
Therefore, we see that Theorem \ref{Thm Main} reduces to the sharp lower bound of first nonzero eigenvalue of the Laplacian by Zhong and Yang \cite{ZY84} if $p=2$ and $f\equiv 0$, and to the sharp lower bound of first nonzero eigenvalue of the $p$-Laplacian by Valtorta \cite{Valtorta12} if $f \equiv 0$. Moreover, when $\kappa =0$, the equality in \eqref{eigenvalue estimate} is achieved when $M$ is one-dimensional circle (when $M$ has no boundary) or the line segment $[-D/2, D/2]$ (when $M$ has boundary).  

For general $\kappa \in \R$, 
the $p=2$ case of Theorem \ref{Thm Main} was due to Bakry and Qian \cite{BQ00}. Andrews and Ni \cite{AN12} gave a simple alternative proof using the modulus of continuity approach and they also demonstrated the sharpness in all dimensions by constructing a sequence of Riemannian manifolds collapsing to the interval $[-D/2,D/2]$. We refer the reader to \cite{AC13, BQ00, LW19eigenvalue, NV14} and the references therein for more historical developments and other related results.

In \cite{LW19eigenvalue}, we used the method of modulus of continuity estimates to give a simple proof of Theorem \ref{Thm Main} for $1<p \leq 2$ and $\kappa \in \R$. 
Recent years have witnessed the great success of this approach and sharp eigenvalues estimates were obtained in \cite{AC11, AC13, AN12, Ni13, SWW19, WZ17}. However, it seems difficult to handle the $p>2$ case with this method, as pointed out in the wonderful survey by Andrews \cite[Section 8]{Andrewssurvey15}.  
The proof presented below uses the classical gradient estimates method initiated by Li \cite{Li79} and Li and Yau \cite{LY80}, which indeed works for all $1<p<\infty$ and $\kappa \in \R$. This approach has been used successfully by various authors to estimate eigenvalues, including \cite{BQ00, Kroger92, Li79, LY80, LW19eigenvalue, NV14, Valtorta12, ZY84}.

In order to handle the $\kappa >0$ case, we have to overcome two difficulties in this paper. 
The sharp gradient comparison theorem (see Theorem \ref{Thm grad comp}), which is the most important ingredient, was only proved for $\kappa \leq 0$ in \cite{LW19eigenvalue} using a two-point maximum principle argument. 
%The most important ingredient 
In the present paper, we use a Bochner type formula for the weighted $p$-Laplacian (see Proposition \ref{Bochner}) to prove the sharp gradient comparison theorem.  
Despite its technicality, this approach works regardless of the sign of $\kappa$, and the value of $p \in (1,\infty)$. 
The second key estimate that we shall establish in this paper is Proposition \ref{a priori bound}, which says that $\l_{p,f}$ is bounded from below by the first nonzero eigenvalue of the eigenvalue problem \eqref{ODE} on the real line. 
This bound is trivial when $\kappa \leq 0$ (the first eigenvalue of the real line is zero in this case), but it is nontrivial when $\kappa >0$ and turns out to be a consequence of the sharp gradient comparison theorem and the compactness of the manifold. 
Once we have these two key ingredients, the rest of proof consists of careful analysis of the one-dimensional models and standard comparison arguments. 

The paper is organized as follows. We derive the Bochner formula for $\Delta_{p,f}$ in Section 2 and use it to prove the sharp gradient comparison theorem in Section 3. 
In Section 4, we study the qualitative behavior of solutions to the ODE \eqref{ODE} and construct one-dimensional models for the purpose of applying the gradient comparison theorem effectively. 
The proof of Theorem \ref{Thm Main} is then given in Section 5. Finally, examples are constructed in Section 6 to demonstrate the sharpness of Theorem \ref{Thm Main}.

\section{Bochner Formula for the Weighted $p$-Laplacian}
The goal of this section is to establish the Bochner formula for $\Delta_{p,f}$ in Proposition \ref{Bochner}, which will play a key role in proving the sharp gradient comparison theorem in Section 3. This is a slight extension of \cite[Section 3]{Valtorta12}, where the author proves the Bochner formula for the $p$-Laplacian.

Let $H_u$ denote the Hessian of $u$ and set 
$$A_u =\frac{H_u(\nabla u, \nabla u)}{|\nabla u|^2}$$ whenever $|\nabla u|\neq 0$. 
The linearization of $\Delta_{p,f}$ near a function $u$ is given by 
\begin{eqnarray*}
P_{u,f}(\eta)&:=& \left.\frac{d}{dt}\right|_{t=0} \Delta_{p,f}(u+t\eta) \\
&=& |\nabla u|^{p-2} \Delta_f \eta +(p-2)|\nabla u|^{p-4}H_\eta(\nabla u, \nabla u) +(p-2)\Delta_{p,f} u \frac{\langle \nabla u, \nabla \eta \rangle}{|\nabla u|^2} \\
&& + 2(p-2)|\nabla u|^{p-4} H_u\left(\nabla u, \nabla \eta -\frac{\nabla u}{|\nabla u|} \left\langle \frac{\nabla u}{|\nabla u|}, \nabla \eta \right\rangle \right) 
\end{eqnarray*}
wherever $|\nabla u|\neq 0$. It is easy to see that this operator is strictly elliptic where the gradient of $u$ does not vanish. 
Let $P^{II}_{u,f}(\eta)$ be defined by 
$$P^{II}_{u,f}(\eta):= |\nabla u|^{p-2} \Delta_f \eta +(p-2)|\nabla u|^{p-4}H_\eta(\nabla u, \nabla u).$$

\begin{prop}[Bochner formula for $\Delta_{p,f}$]\label{Bochner}
Let $x\in M$ and $U\subset M$ be an open set containing $x$. For any function $u\in C^3(U)$ with $\nabla u \neq 0$ on $U$, we have at $x$, 
\begin{eqnarray*}
\frac{1}{p}P^{II}_{u,f}(|\nabla u|^p) &=&  |\nabla u|^{p-2} \langle \nabla \Delta_{p,f} u, \nabla u \rangle -(p-2)|\nabla u|^{p-2} A_u \Delta_{p,f}u \\
&& +|\nabla u|^{2(p-2)} \left(|H_u|^2 +p(p-2)A_u^2 +\Ric_f(\nabla u,\nabla u)  \right),
\end{eqnarray*}
where $\Ric_f =\Ric +\nabla^2 f.$
\end{prop}
\begin{proof}
Let $P^{II}_u(\eta):= |\nabla u|^{p-2} \Delta \eta +(p-2)|\nabla u|^{p-4}H_\eta(\nabla u, \nabla u)$ be the second order part of $P_{u,f}(\eta)$. The Bochner formula for the $p$-Laplacian, shown in \cite[Proposition 3.1]{Valtorta12}, states that 
\begin{eqnarray*}
\frac{1}{p}P^{II}_{u}(|\nabla u|^p) 
&=&  |\nabla u|^{p-2} \langle \nabla \Delta_{p} u, \nabla u \rangle -(p-2)|\nabla u|^{p-2} A_u \Delta_{p}u \\
&& +|\nabla u|^{2(p-2)} \left(|H_u|^2 +p(p-2)A_u^2 +\Ric(\nabla u,\nabla u)  \right).
\end{eqnarray*}
We then compute, using the identities $\Delta_p u = \Delta_{p,f} u +|\nabla u|^{p-2} \langle \nabla u, \nabla f \rangle$ and $\Ric =\Ric_f -\nabla^2 f$, that 
\begin{eqnarray*}
\frac{1}{p}P^{II}_{u}(|\nabla u|^p)  &=& %\frac{1}{p}P^{II}_{u}(|\nabla u|^p) - |\nabla u|^{p-2} \langle \nabla f, \nabla |\nabla u|^p \rangle\\
|\nabla u|^{p-2} \langle \nabla \Delta_{p,f} u, \nabla u \rangle -(p-2)|\nabla u|^{p-2} A_u \Delta_{p,f}u \\
&& +|\nabla u|^{2(p-2)} \left(|H_u|^2 +p(p-2)A_u^2 +\Ric_f(\nabla u,\nabla u)  \right) \\
&& + |\nabla u|^{p-2} \left\langle \nabla (|\nabla u|^{p-2} \langle \nabla u, \nabla f \rangle), \nabla u \right\rangle \\
&& -(p-2)|\nabla u|^{2(p-2)} A_u \langle \nabla u, \nabla f \rangle \\
&& -|\nabla u|^{2(p-2)} \left(\nabla^2 f (\nabla u,\nabla u)  \right) \\
&=& |\nabla u|^{p-2} \langle \nabla \Delta_{p,f} u, \nabla u \rangle -(p-2)|\nabla u|^{p-2} A_u \Delta_{p,f}u \\
&& +|\nabla u|^{2(p-2)} \left(|H_u|^2 +p(p-2)A_u^2 +\Ric_f(\nabla u,\nabla u)  \right) \\
&& +\frac{1}{p} |\nabla u|^{p-2} \langle \nabla |\nabla u|^p, \nabla f\rangle.
\end{eqnarray*}
The desired formula then follows from 
\begin{eqnarray*}
\frac{1}{p}P^{II}_{u,f}(|\nabla u|^p)  &=& \frac{1}{p}P^{II}_{u}(|\nabla u|^p) - \frac{1}{p}|\nabla u|^{p-2} \langle \nabla f, \nabla |\nabla u|^p \rangle.
\end{eqnarray*}
\end{proof}

\section{The Gradient Comparison Theorem}

In this section, we use the Bochner formula for $\Delta_{p,f}$ proved in Proposition \ref{Bochner} to derive sharp gradient comparison theorems for eigenfunctions of $\Delta_{p,f}$. We emphasis that the proof works for all $1<p<\infty$ and $\kappa \in \R$. 

\begin{thm}\label{Thm grad comp}
Let $(M^n,g,f)$ be a compact Bakry-\'Emery manifold (possibly with $C^2$ convex boundary) satisfying $\Ric +\nabla^2 f \geq \kappa\, g$ for some $\kappa \in \R$. 
%Suppose that $u$ is an eigenfunction of $\Delta_{p,f}$ with eigenvalue $\l$:
Let $u$ be a solution of 
\begin{equation}
    \Delta_{p,f} u=-\l |u|^{p-2} u,
\end{equation}
normalized so that $-1 =\min \{u\} < 0 < \max \{u\} \leq 1 $ (in case $\p M \neq \emptyset$, we assume that $u$ satisfies Neumann boundary condition). 
Suppose $w:[a,b] \to \R$ is a solution of the one-dimensional equation:
\begin{equation}
    (p-1)|w'|^{p-2}w'' -\kappa\, t  |w'|^{p-2}w' =-\lambda |w|^{p-2} w
\end{equation}
which is strictly increasing on $[a,b]$, and such that the range of $u$ is contained in $[w(a),w(b)]$. Then we have for all $x\in M$, 
\begin{equation*}
    |\nabla u(x)| \leq w'\left(w^{-1}(u(x)) \right).
\end{equation*}
\end{thm}

\begin{proof}
\textbf{Case 1:} $\p M = \emptyset$. \\
We may assume that $[\min\{u\} , \max \{u\}] \subset (w(a), w(b))$, because otherwise we can replacing $u$ by $\a u$ and then letting $\a \to 1^{-}$.
Let $w:=w_{\eps}(t)$ be the solution of the initial value problem
\begin{equation}\label{IVP8.1}
    \begin{cases}
    (p-1)|w'|^{p-2}w'' -(\kappa -\eps) \, t |w'|^{p-2}w'+\lambda |w|^{p-2}w =0, \\
    w(a)=-1, w'(a)=0,
    \end{cases}
\end{equation}
which is strictly increasing on $[a, b_\eps]$ with $\lim_{\eps \to 0} b_{\eps} =b$. 
For $\eps >0$ sufficiently small, we still have that 
$ [\min \{u \} , \max\{ u\} ] \subset (w(a), w(b_\eps)) $. 

Let $c_0 \geq 1$ be the number defined by 
\begin{equation*}
    c_0=\inf\{ c \geq 1:  Z_c(x) =|\nabla u(x)|^p - \left. (cw')^p \right|_{ (cw)^{-1}(u(x)) } \leq 0 \text{ on } M\}.
\end{equation*}
Clearly $c_0$ is finite since $Z_c(x)$ approaches negative infinity uniformly as $c$ approaches infinity. By continuity, there exists $x_0 \in M$ such that 
$$0=Z_{c_0}(x_0) =\max_{x\in M}Z_{c_0}(x).$$
For simplicity of notations, we use $t=(c_0w)^{-1}(u(x))$ as an intermediate variable and write $\vp(t) =(c_0w)(t)$. 
Thus the function $Z_{c_0}(x) = |\nabla u|^p -(\vp')^p$ attains zero maximum at $x_0$. 
Note that $|\nabla u(x_0)| \neq 0$ since $\vp$ has positive derivative.

Hereafter, we assume that $u\in C^{3}(U)$ for some open neighborhood $U$ of $x_0$. This is certainly the case if $u(x_0)\neq 0$ or if $p\geq 2$, as pointed out in Remark \ref{rmk regularity}. In case $1<p<2$ and $u(x_0)=0$, then $u$ has only $C^{2,\a}$ regularity near $x_0$. However, this regularity issue is not an obstacle to the argument, as we will explain in Remark \ref{Rmk regularity}. Here and below, all the derivatives of $\vp$ are evaluated at $t_0=(c_0w)^{-1}(u(x_0))$. 

The first derivative test implies that $\nabla Z_{c_0}(x_0) =0$, which produces the following identity at $x_0$,
\begin{equation*}
    p|\nabla u|^{p-2}H_u \nabla u =\frac{p}{p-1} \L_p \vp \nabla u,
\end{equation*}
where $\L_p \vp := (p-1)(\vp')^{p-2}\vp''$ is the one-dimensional $p$-Laplacian. 
In particular, we have  
\begin{equation}\label{eq8.2}
    (p-1)|\nabla u|^{p-2} A_u =\L_p \vp. 
\end{equation}

Next we calculate and estimate the second derivatives. 
Using Proposition \ref{Bochner}, we obtain that at $x_0$,
\begin{eqnarray}\label{eq8.3} \nonumber
\frac{1}{p} P^{II}_{u,f}(|\nabla u|^p) &=& |\nabla u|^{p-2} \langle \nabla \Delta_{p,f} u, \nabla u \rangle -(p-2)|\nabla u|^{p-2} A_u \Delta_{p,f}u \\ \nonumber
&& +|\nabla u|^{2(p-2)} \left(|H_u|^2 +p(p-2)A_u^2 +\Ric_f(\nabla u,\nabla u)  \right) \\ \nonumber
&\geq & -\l |\nabla u|^{p-2} \langle \nabla \left( |u|^{p-2} u \right), \nabla u \rangle -(p-2)\l |u|^{p-2}u |\nabla u|^{p-2} A_u  \\ \nonumber
&& +|\nabla u|^{2(p-2)} \left((p-1)^2A_u^2 +\kappa |\nabla u |^2  \right) \\ \nonumber
&= & -(p-1) \l u^{p-2} |\nabla u|^{p} -\frac{p-2}{p-1}\l |u|^{p-2}u \L_p\vp \\
&& + (\L_p\vp)^2 +\kappa |\nabla u |^{2p-2},
\end{eqnarray}
where we have used $|H_u|^2 \geq A_u^2$ and $\Ric_f(\nabla u, \nabla u) \geq \kappa |\nabla u|^2$ in the inequality, and \eqref{eq8.2} in getting the last equality. 

On the other hand, direct calculation shows 
\begin{eqnarray*}
\frac{1}{p}\nabla (\vp')^p &=& \frac{1}{p-1}\L_p\vp \nabla u,\\
\frac{1}{p} \Delta_f  (\vp')^p &=& \frac{1}{p-1} \L_p\vp \Delta_f u +\frac{1}{p-1} \frac{d}{dt} (\L_p\vp) \frac{1}{\vp'} |\nabla u|^2,\\
\frac{1}{p} H_{(\vp')^p} (\nabla u, \nabla u) &=& \frac{1}{p-1}\L_p\vp A_u |\nabla u|^2 + \frac{1}{p-1} \frac{d}{dt} (\L_p\vp) \frac{1}{\vp'} |\nabla u|^4
\end{eqnarray*}
Putting the above identities together, we obtain that 
\begin{eqnarray}\label{eq8.4} \nonumber
\frac{1}{p} P^{II}_{u,f}((\vp')^p) 
&=&\frac{1}{p}|\nabla u|^{p-2} \Delta_f (\vp')^p +\frac{1}{p}(p-2)|\nabla u|^{p-4}H_{(\vp')^p} (\nabla u, \nabla u) \\ \nonumber 
&=& \frac{\L_p\vp}{p-1} |\nabla u|^{p-2}\left(\Delta_f u +A_u \right) +|\nabla u|^p \frac{d}{dt} (\L_p\vp)\frac{1}{\vp'} \\
&=& \frac{-\l |u|^{p-2} u}{p-1} \L_p\vp +|\nabla u|^p \frac{d}{dt} (\L_p\vp)\frac{1}{\vp'},
\end{eqnarray}
where we have used $|\nabla u|^{p-2}\left(\Delta_f u +A_u \right) =\Delta_{p,f} u =-\l |u|^{p-2} u$ in the last equality. 

The second derivative test implies $P^{II}_{u,f} (Z_{\bar{c}}) \leq 0 $ at $x_0$, thus we have by \eqref{eq8.3} and \eqref{eq8.4} that 
\begin{eqnarray}\label{eq8.5} \nonumber
0 &\geq& \frac{1}{p} P^{II}_{u,f} (Z_{\bar{c}}) = \frac{1}{p} P^{II}_{u,f}(|\nabla u|^p) - \frac{1}{p} P^{II}_{u,f}((\vp')^p)\\ \nonumber
&\geq & -(p-1) \l u^{p-2} |\nabla u|^{p} -\frac{p-2}{p-1}\l |u|^{p-2}u \L_p\vp \\ 
&& + (\L_p\vp)^2 +\kappa |\nabla u |^{2p-2} + \frac{\l |u|^{p-2} u}{p-1} \L_p\vp - |\nabla u|^p \frac{d}{dt} (\L_p\vp)\frac{1}{\vp'}.
\end{eqnarray}
Substituting $|\nabla u(x_0)|^p =(\vp'(t_0))^p $ and $u(x_0) =\vp(t_0)$ into \eqref{eq8.5}
gives that at $t_0$, 
\begin{eqnarray}\label{eq8.6} \nonumber
0 &\geq & -(p-1) \l \vp^{p-2} (\vp')^{p} +\frac{(p-2)}{p-1} \l |\vp|^{p-2}\vp \L_p \vp \\ \nonumber
&& +(\L_p\vp)^2 +\kappa (\vp')^{2p-2}   + \frac{\l |\vp|^{p-2}\vp}{p-1} \L_p\vp - (\vp')^{p-1} \frac{d}{dt} (\L_p\vp)\\ \nonumber
&\geq & \L_p\vp \left(\L_p \vp -(\kappa-\eps) t (\vp')^{p-1} +\l |\vp|^{p-2}\vp \right) + \eps (\vp')^{2p-2}  \\ 
&&  - (\vp')^{p-1} \left(\frac{d}{dt} (\L_p\vp)
+  (p-1) \l \vp^{p-2} \vp' - (\kappa -\eps)t \, \L_p \vp - (\kappa -\eps) (\vp')^{p-1} \right) 
\end{eqnarray}
Since $\vp$ satisfies the ODE
$$\L_p \vp -(\kappa -\eps) t (\vp')^{p-1} + \l |\vp|^{p-2}\vp =0,$$
it follows that $\L_p\vp$ satisfies  
$$\frac{d}{dt} (\L_p\vp)-(\kappa -\eps) (\vp')^{p-1} -(\kappa -\eps) t \, \L_p\vp +(p-1)\l \vp^{p-2}\vp' =0.$$
We then easily get from \eqref{eq8.6} that 
$$ 0 \geq \eps (\vp')^{2p-2} > 0, $$
which is clearly a contradiction. 
The desired gradient estimates follows immediately since the solutions $w_\eps$ converges in $C^1$ to the solution $w$ as $\eps \to 0$. 

\textbf{Case 2:} $\p M \neq \emptyset$ and $u$ satisfies the Neumann boundary condition. \\
The proof for Case 1 remains valid as long as $x_0$ is in the interior of $M$.   
If $x_0 \in \p M$, we follow the argument in \cite[Lemma 18]{NV14} to show that $\nabla Z_{c_0}(x_0) =0$. 
Once this is established, it then follows that $P^{II}_{u,f} (Z_{c_0}(x_0)) \leq 0$ and the rest of the proof proceeds as in Case 1. 
Thus it suffices to show the following claim.

\textbf{Claim: } The equation $\nabla Z_{c_0}(x_0) =0$ remains valid even if $x_0 \in \p M$.  
\begin{proof}[Proof of Claim]
Let $\nu$ be the outward unit normal vector field of $\p M$. Since $Z_{c_0}(x)$ attains its maximum at $x_0 \in \p M$, we know that all tangential derivatives of $Z_{c_0}(x)$ vanish at $x_0$ and
\begin{eqnarray*}
   0 \leq  \langle \nabla Z_{c_0}, \nu \rangle(x_0)&=& p|\nabla u|^{p-2} H_u(\nabla u, \nu) - p (\vp')^{p-2}\vp'' \langle \nabla u, \nu \rangle \\
   &=& p|\nabla u|^{p-2} H_u(\nabla u, \nu) \\
   &=& -p |\nabla u|^{p-2} II(\nabla u, \nabla u) \leq 0.
\end{eqnarray*}
Here the last step is because of the convexity of $\p M$. Therefore we have $\nabla Z_{c_0}(x_0) =0$ and the claim is proved. 
\end{proof}

The proof of Theorem \ref{Thm grad comp} is complete now. 
\end{proof}

\begin{remark}\label{rmk regularity}
The eigenfunction $u$ is in general not smooth. We have $u \in C^{1,\a}(M) \cap W^{1,p}(M)$, and elliptic theory ensures that $u$ is smooth where $\nabla u\neq 0$ and $u\neq 0$.
If $\nabla u \neq 0$ and $u(x)=0$, then $u \in C^{3,\a}(U)$ if $p>2$ and $u \in C^{2,\a}(U)$ if $1<p<2$, where $U$ is a small neighborhood of $x$.
We refer the reader to \cite{Tolksdorf84} for these results.
\end{remark}

\begin{remark}\label{Rmk regularity}
If $1<p<2$ and $u(x_0) =0$, we only know that $u$ is $C^{2, \a}$ near $x_0$ and $Z$ is $C^{1,\a}$ near $x_0$. Thus the $P^{II}_{u,f}(Z)$ may not be defined since there are two diverging terms in it. 
As we can see in equation \eqref{eq8.5}, these term are 
\begin{equation*}
    -(p-1)\l |u|^{p-2} |\nabla u|^p \text{  and  } -|\nabla u|^p \frac{1}{\vp'} \frac{d}{dt} (\L_p \vp).
\end{equation*}
Since $\nabla u(x_0) \neq 0$, there exists an open set $U$ containing $x_0$ such that $U\setminus \{u=0\}$ is open and dense in $U$. On this set, we see that these two terms exactly cancel each other, and all the other term in $P^{II}_{u,f}(Z_{c_0})$ are well-defined and continuous on $U$. Thus the formula $P^{II}_{u,f} Z_{c_0} \leq 0$ is valid even in this low regularity setting.  
\end{remark}

\section{One-dimensional Models for $\kappa >0$}

In this section, we study the qualitative behavior of the one-dimensional equation 
\begin{equation}\label{ODE new}
     (p-1)|w'|^{p-2}w'' -\kappa \, t |w'|^{p-2}w'+\lambda |w|^{p-2}w =0.
\end{equation}
for $\kappa >0$.  The case $\kappa \leq 0$ was treated in \cite{LW19eigenvalue}. 

We recall some basic definitions and properties of $p$-trigonometric functions and refer the reader to \cite[Chapter 1]{DP05} for more details. For $1<p<\infty$, let $\pi_p$ be the positive number defined by
\begin{equation*}
    \pi_p=\int_{-1}^1 \frac{ds}{(1-s^p)^{1/p}} =\frac{2\pi}{p \sin(\pi/p)}.
\end{equation*}
The $p$-sine function $\sin_p:\R \to [-1,1]$ is defined implicitly on $[-\pi_p/2,3\pi_p/2]$ by
\begin{equation*}
    \begin{cases}
    t=\int_0^{\sin_p(t)} \frac{ds}{(1-s^p)^{1/p}} &  \text{ if } t\in [-\frac{\pi_p}{2},\frac{\pi_p}{2}], \\
    \sin_p(t)=\sin_p(\pi_p-t) & \text{ if }  t\in [\frac{\pi_p}{2}, \frac{3\pi_p}{2}],
    \end{cases}
\end{equation*}
and is periodic on $\R$ with period $2\pi_p$. It's easy to see that for $p\neq 2$ this function is smooth around noncritical points, but only $C^{1,\alpha}(\R)$ with $\a =\min\{p-1,(p-1)^{-1} \}$.
By defining
\begin{equation*}
 \cos_p(t)=\frac{d}{dt} \sin_p(t) \text{ and } \tan_p(t) =\frac{\sin_p(t)}{\cos_p(t)},
\end{equation*}
we then have the following generalized trigonometric identities:
\begin{align*}
   & |\sin_p(t)|^p+|\cos_p(t)|^p =1, \\
   &    \frac{d}{dt}\tan_p(t) =\frac{1}{|\cos_p(t)|^p} =1+|\tan_p(t)|^p, \\
   & \frac{d}{dt}\arctan_p(t) =\frac{1}{1+|t|^p}.
\end{align*}
Let $\alpha=\left(\frac{\lambda}{p-1}\right)^{1/p}$.
We introduce the $p$-polar coordinates $r$ and $\theta$ defined by
\begin{equation}
    \alpha w =r \sin_{p}(\theta), \; w'=r \cos_p(\theta),
\end{equation}
or equivalently,
\begin{equation}
    r=\left( (w')^p +\alpha^p w^p \right)^{\frac 1 p}, \; \theta=\arctan_p\left(\frac{\alpha w}{w'}\right).
\end{equation}
If $w$ is a solution of \eqref{ODE new} with $w(a)=-1$ and $w'(a)=0$, 
then direct calculation shows that $\theta$ and $r$ satisfy

\begin{align}\label{IVP 3}
    & \begin{cases}
    \theta'=\alpha -\frac{\kappa t}{p-1} \cos_p^{p-1}(\theta)\sin_p(\theta), \\
    \theta(a) = -\frac{\pi_p}{2}, \; ( \text{mod} \, \pi_p);
    \end{cases} \\ \label{IVP 4}
    & \begin{cases}
    \frac{d}{dt}\log r=\frac{\kappa t}{p-1} \cos_p^{p}(\theta),  \\
    r(a) = \alpha.
   \end{cases}
\end{align}
Since both $\sin_p(t)$ and $\cos_p^{p-1}(t)$ are Lipschitz functions with Lipschitz constant $1$, we can apply Cauchy's theorem to obtain existence, uniqueness, and continuous dependence on the parameters for \eqref{ODE new}, \eqref{IVP 3} and \eqref{IVP 4}. Indeed, we have the following proposition.
\begin{prop}\label{prop Cauchy}
For any $a \in \R$, there exists a unique solution $w$ to \eqref{ODE new} with $w(a)=-1$ and $w'(a) =0$, defined on $\R$ with $w, (w')^{p-2}w' \in C^1(\R)$. Moreover, the solution depends continuously on the parameters in the sense of local uniform convergence of $w$ and $\dot{w}$ in $\R$.
\end{prop}

Let $\l_0$ be the first nonzero eigenvalue of the eigenvalue problem \eqref{ODE new} on the real line, i.e., 
\begin{equation}\label{lambda 0 def}
    \l_0 =\inf\left\{ \frac{\int_\R |\vp'|^p e^{-\frac{\kappa}{2} s^2} ds}{\int_\R |\vp|^p e^{-\frac{\kappa}{2} s^2} ds}: \vp \in W^{1,p}(\R, e^{-\frac{\kappa}{2}s^2}ds)\setminus \{0\}, \int_\R |\vp|^{p-2} \vp e^{-\frac{\kappa}{2} s^2} ds =0 \right\}.
\end{equation}
Let's first discuss the special case $p=2$. In this case, equation \eqref{ODE new} (when normalizes so that $\kappa =1$) is 
the so-called Hermite's differential equation and solutions with polynomial growth are given by Hermitian polynomials. In particular,  we have $\l_0 =\kappa$ and the corresponding eigenfunction is $\kappa \, t$. On the other hand, it is known that the first nonzero eigenvalue of the $f$-Laplacian is bounded from below by $\kappa$ if $\Ric +\nabla^2 f \geq \kappa \, g$, see for example \cite[Theorem 1.6]{LW19eigenvalue}. 
Thus when $p=2$, we have $\lambda_{p,f} \geq \l_0$ (strict inequality when $M$ is compact) and this inequality plays an important role in the construction of one-dimensional models in \cite{BQ00}.

Our first step here is to get the non-sharp bound $\l_{p,f} > \l_0$ for all $1<p<\infty$ and $\kappa >0$. 
It turns out that this is a consequence of the 
sharp gradient estimates in Theorem \ref{Thm grad comp} 
and the compactness of $M$.

\begin{prop}\label{a priori bound}
Let $(M^n,g,f)$ be a compact Bakry-\'Emery manifold (possibly with $C^2$ convex boundary) satisfying $\Ric +\nabla^2 f \geq \kappa\, g$ for some $\kappa >0$. 
%Suppose that $u$ is an eigenfunction of $\Delta_{p,f}$ with eigenvalue $\l$:
Let $\l_{p,f}$ be the first nonzero eigenvalue of $\Delta_{p,f}$ (with Neumann boundary condition if $\p M \neq \emptyset$). 
Then 
\begin{equation*}
    \l_{p,f} > \l_0,
\end{equation*}
where $\l_0$ is defined in \eqref{lambda 0 def}.
\end{prop}

We first prove an elementary lemma.  
\begin{lemma}\label{lemma 3.1}
For any $\l \leq \l_0$, the ODE
\begin{equation}\label{eq8.7}
    (p-1)|w'|^{p-2}w'' -\kappa t |w'|^{p-2}w' +\l |w|^{p-2}w =0
\end{equation}
admits an odd solution $w:\R \to \R$ satisfying $w'(t)>0$ for all $t\in \R$. 
\end{lemma}

\begin{proof}[Proof of Lemma \ref{lemma 3.1}]
By Proposition \ref{prop Cauchy}, there exists a solution $w$ of \eqref{eq8.7} satisfying $w(0)=0$ and $w'(0)=1$. 
The oddness of $w$ follows from uniqueness of solutions. The condition $w'(t)>0$ for all $t\in \R$ is a consequence of $\l \leq \l_0$. Otherwise, we get a solution of \eqref{eq8.7} on some interval $[-a, a]$ with $w'(-a)=w'(a)=0$ and $w'(t) >0$ on $(-a,a)$, implying that $\l$ is the first nonzero eigenvalue of \eqref{eq8.7} on $[-a,a]$, which would contradict $\l \leq \l_0$. 

\end{proof}

\begin{proof}[Proof of Proposition \ref{a priori bound}]
We argue by contradiction. 
Suppose $\l \leq \l_0$, then by Lemma \ref{lemma 3.1}, there exists a odd solution $v:\R \to \R$ of the ODE \eqref{eq8.7} satisfying 
$v'(t)>0$ for all $t\in \R$. 
Since $v(t)$ is strictly increasing, we have that either $v(t)$ approaches infinity as $t\to \infty$ or $\lim_{t\to \infty} v(t) = A \in (0, \infty)$. 

Let $u$ be an eigenfunction associated to the eigenvalue $\l_{p,f}$ normalized so that $-1=\min \{u\} < 0 < \max \{u\} \leq 1$. 
Consider two points $x_0$ and $y_0$ such that $u(x_0)=\min_{x\in M} u(x)$, and $u(y_0)=\max_{x\in M} u(x)$. Let $L=d(x_0,y_0)$ and $\gamma: [0, L] \to M$ be a unit speed geodesic joining $x_0$ and $y_0$. 
Define $h(t)=u(\gamma(t))$ and choose $I \subset [0,L]$ such that 
%$I \subset (f')^{-1}(0,\infty)$ 
$h'>0$ on $I$ and $h^{-1}$ is well-defined in a subset of full measure of $[-1, \max\{u\}]$.  

For any number $c$ such that the range of $cv$ contains the interval $[-1,1]$, we can apply Theorem \ref{Thm grad comp} 
with $w= cv$ on the interval $[w^{-1}(-1), w^{-1}(\max \{u\})]$ to conclude that 
\begin{eqnarray*}
D &\geq& \int_0^L \, dt \geq \int_{I} \, dt \geq \int_{-1}^{\max \{u\}} \frac{dy}{h'(h^{-1}(y))} 
\geq \int_{-1}^{\max \{u\}} \frac{dy}{w'(w^{-1}(y))} \\
&=& \int_{w^{-1}(-1)}^{w^{-1}(\max\{u\})} dt  = w^{-1}(\max\{u\})-w^{-1}(-1) \geq -w^{-1}(-1)
=v^{-1}\left(\frac{1}{c}\right).
\end{eqnarray*}
If $\lim_{t\to \infty} v(t) =\infty$, we see that the right hand side goes to infinity by letting $c\to 0$.  
If $\lim_{t\to \infty} v(t) =A \in (0,\infty)$, the right hand side goes to infinity by letting $c$ decrease to $A^{-1}$. Either way, this is a contradiction to the finiteness of the diameter of $M$. 
\end{proof}

For the purpose of getting sharp eigenvalue estimates, we need to show that for any eigenfunction $u$ of $\Delta_{p,f}$ with eigenvalue $\l >\l_0$, there exist an interval $[a,b]$ and a solution $w$ of \eqref{ODE new} such that $w$ is strictly increasing on $[a,b]$ with $w(a)=\min\{ u\}$ and $w(b)=\max\{ u\}$.
As we shall see, this can be achieved by varying the initial data.

The rest of this section is a slight modification of \cite[Section 4]{LW19eigenvalue}. 
Fix $\l > \l_0$ and $\kappa >  0$.
Let $w_a$ be the solution of the following initial value problem (IVP)
\begin{equation}\label{IVP new}
    \begin{cases}
    (p-1)|w'|^{p-2}w'' -\kappa \, t |w'|^{p-2}w'+\lambda |w|^{p-2}w =0, \\
    w(a)=-1, w'(a)=0,
    \end{cases}
\end{equation}
with $a\in \R$.

\begin{prop}\label{prop symmetric}
Fix $\l > \l_0$ and $\kappa > 0$. There exists a unique $\bar{a}>0$ such that the solution $w_{-\bar{a}}$ to the IVP \eqref{IVP new} is odd. In particular, $w_{-\bar{a}}$ restricted to $[-\bar{a},\bar{a}]$ has nonnegative derivative and has maximum value equal to one.
\end{prop}

\begin{proof}
Consider the initial value problem 
\begin{equation}\label{IVP 5}
    \begin{cases}
    \dot{\theta}=\alpha -\frac{\kappa t}{p-1} \cos_p^{p-1}(\theta)\sin_p(\theta), \\
    \theta(0) = 0.
    \end{cases}
\end{equation}
The uniqueness of solutions implies that $\theta(t)$ is an odd function. 
The fact $\l >\l_0$ implies there exists $\bar{a} >0$ such that 
$\theta(\bar{a}) =\pi_p /2$. 
It's easily seen that the corresponding solution $r(t)$ to \eqref{IVP 4} is even, regardless of its initial value.
The proposition follows by translating obtained information on $\theta$ and $r$ back to $w$.
\end{proof}

For the solution $w=w_a$ of \eqref{IVP new}, we define
\begin{align*}
    b(a) &=\inf \{ b>a : w'(b)=0 \}, \\
   m(a) &= w_a (b(a)), \\
   \delta(a) & =b(a)-a.
\end{align*}
In other words, $b(a)$ is the first value $b>a$ such that $w'(b)=0$ with the convention that $b(a)=\infty$ if such a value does not exist.  Thus $w$ is strictly increasing on the interval $[a, b(a)]$ and $m(a)$ is the maximum of $w$ on $[a,b(a)]$.
The function $\delta(a)$ measures the length of the interval where $w(a)$ increases from $-1$ to $m(a)$. 

We are concerned with the range of the function $m(a)$ as $a$ varies.
Clearly $m(a) >0$ since the eigenfunction $w$ changes sign. Also, Proposition \ref{prop symmetric} implies $m(-\bar{a}) =1$. We shall show in the next proposition that $m(a)$ goes to zero as $a$  goes to $-\infty$. It then follows from the Intermediate Value Theorem that the range of $m(a)$ covers $(0,1]$ when $a$ varies in $ (-\infty, -\bar{a}]$.

\begin{prop}\label{prop range}
$$\lim_{a \to -\infty} m(a) =0. $$
\end{prop}
\begin{proof}
The argument is dual to the proof of \cite[Proposition 4.3]{LW19eigenvalue}, and we omit it here.
\end{proof}

\begin{prop}\label{prop central}
We have $\delta (a) \geq \delta(-\bar{a})$ for all $a \leq -\bar{a}$ with strict inequality if $a \neq -\bar{a}$.
\end{prop}
\begin{proof}
The argument is dual to the proof of \cite[Proposition 4.4]{LW19eigenvalue}, and we omit it here.
\end{proof}

At last, we study $\bar{\delta} :=\delta(-\bar{a}) =2\bar{a}$ as a function of $\l$, having fixed $p$ and $\kappa$.
%, $1<p< \infty$ and $\kappa <0$.
It's easy to see that $\bar{\delta}$ is a strictly decreasing function and so invertible. Thus we can define its inverse $\l(\delta)$, which is a continuous and decreasing function. Moreover, it can be characterized in the following equivalent way.
\begin{prop}\label{prop4.5}
For fixed $\kappa >0, 1<p<\infty$, we have that given  $\delta>0$, $\l$ is the first nonzero Neumann eigenvalue of the one-dimensional problem
\begin{equation*}
    (p-1)|w'|^{p-2}w'' -\kappa \, t |w'|^{p-2}w' =-\l |w|^{p-2}w
\end{equation*}
on $[-\delta/2,\delta/2]$.
\end{prop}

\section{Proof of Theorem \ref{Thm Main}}
After all the preparations in the previous two sections, we finally prove Theorem \ref{Thm Main} in this section.
\begin{proof}[Proof of Theorem \ref{Thm Main}]
It suffices to prove the case $\kappa >0$, as the case $\kappa \leq 0$ was proved in \cite{LW19eigenvalue}. 
%It suffices to prove the case $\kappa<0$, as the case $\kappa=0$ follows by letting $\kappa \to 0$.
Let $u$ be an eigenfunction of $\Delta_{p,f}$ associated to the eigenvalue $\l$. In view of
\begin{equation*}
    \int_M |u|^{p-2} u \, e^{-f} =0,
\end{equation*}
we can normalize $u$ so that $\min \{ u\} =-1$ and $\max\{ u\} \in (0,1]$.

By Proposition \ref{prop symmetric} and \ref{prop range},
there exists an interval $[a,b]$ and a solution $w$ of \eqref{IVP new} such that $w$ is strictly increasing on $[a,b]$ with $w(a) =-1 =\min \{ u\}$ and $w(b) = \max \{ u \} \in (0,1]$. Moreover, we have $\l=\l_{p}([a,b])$, the first nonzero Neumann eigenvalue of the eigenvalue problem $ |w'|^{p-2}w'' -\kappa \, t (p-1)|w'|^{p-2}w'+\lambda |w|^{p-2}w =0$ on $[a,b]$.

Let $x$ and $y$ be such that $u(x)=\min_{M} u$ and $u(y)=\max_M u$.
Consider a unit speed minimizing geodesic $\gamma:[0,d(x,y)]\to M $ joining $x$ and $y$. Let $h(t)=u(\gamma(t))$ and choose $I \subset [0,d(x,y)]$ such that 
%$I \subset (f')^{-1}(0,\infty)$ 
$h'>0$ on $I$ and $h^{-1}$ is well-defined in a subset of full measure of $[-1,u_{\max}]$. Then we get, by change of variables and the sharp gradient estimate proved in Theorem \ref{Thm grad comp}, that 
\begin{eqnarray*}
    D &\geq& \int_0^{d(x,y)} dt \geq \int_I \; dt \geq \int_{-1}^{u_{\max}} \frac{dy}{h'(h^{-1}(y))} \geq \int_{-1}^{w_{\max}} \frac{dy}{w'(w^{-1}(y))} \\
    &=&\int_a^{b(a)} dt  =\delta(a) \geq \delta(\bar{a}),
\end{eqnarray*}
where the last inequality is proved in Proposition \ref{prop central}.
This and Proposition \ref{prop4.5} yield immediately the desired estimate.
\end{proof}

\section{Sharpness}
In this section, we show that the lower bound \eqref{eigenvalue estimate} given in Theorem \ref{Thm Main} is sharp for $n\geq 3$ for any $\kappa$ or $n\geq 2$ for $\kappa \leq 0$. More precisely, for each $\eps >0$, we construct a Bakry-\'Emery manifold $(M,g,f)$ with diameter $D$ and satisfying $\l_{p,f} < \mu_p(\kappa, D) +\eps$. 

For manifolds with boundary, the construction is rather simple. Take a cylinder $rS^{n-1} \times [-D/2, D/2]$ for $r$ sufficiently small, with quadratic potential $f=\frac{\kappa}{2}s^2$. 
It is easy to see this is a Bakry-\'Emery manifold with diameter $D(1+o(r))$ and $\Ric+ \nabla^2 f \geq \kappa \, g$. 
Let $w$ be an eigenfunction associated to the eigenvalue $\mu_p(\kappa, D)$ of the one-dimensional eigenvalue problem \eqref{ODE}
on $[-D/2, D/2]$. 
Substituting the test function $\psi(z,s)=w(s)$ into the Rayleigh quotient yields
\begin{equation*}
    \l_{p,f} \leq \frac{\int_M |\nabla \psi |^p e^{-f} d\mu_g}{\int_M |\psi|^p e^{-f} d\mu_g} =\frac{\int_{-D/2}^{D/2} |w'|^p e^{-f} ds }{\int_{-D/2}^{D/2} |w|^p e^{-f} ds } = \mu_p(\kappa, D)
\end{equation*}
It follows that $\l_{p,f} \to \mu_p(\kappa, D)$ as $r\to 0$, proving the sharpness of the estimate \eqref{eigenvalue estimate} in Theorem \ref{Thm Main}.

To demonstrate the sharpness of \eqref{eigenvalue estimate} in the smaller class of manifolds without boundary, we need a more involved construction. 
The idea is to attach spherical caps to the ends of the above examples. 
The Bakry-E\'mery manifolds are constructed exactly the same as in \cite{AN12}. The Riemannian manifold $M$, which is approximately a thin cylinder with hemispherical caps attached at each end, is constructed as follows. Let $\gamma$ be the curve in $\R^2$ with curvature $k$ given as function of arc length for suitably small $r>0$ and $\delta >0$ small compared ro $r$, by 
\begin{equation}\label{eq6.1}
    k(s) =\begin{cases} \frac{1}{r}, & s\in [0,\frac{\pi r}{2} -\delta], \\
    \vp\left(\frac{s-\frac{\pi r}{2}}{\delta}  \right) \frac{1}{r}, & s \in [\frac{\pi r}{2} -\delta, \frac{\pi r}{2} +\delta], \\
    0, & s\in [\frac{\pi r}{2} +\delta, D],
    \end{cases}
\end{equation}
and extended to be even under reflection in both $s=0$ and $s=D/2$. 
Here $\vp$ is
a smooth nonincreasing function with $\vp(s) = 1$ for $s \leq  -1$, $\vp(s) = 0$ for $s \geq 1$, and
satisfying $\vp(s)+ \vp(-s) = 1$.
Geometrically, this corresponds to a pair of line segments parallel to the $x$ axis, capped by semicircles of radius $r$ and smoothed at the joins. 
Let $(x(s), y(s))$ be the corresponding embedding and we choose the point corresponding to $s=0$ to have $y(0)=0$ and $y'(0)=1$. 
Let $(M^n,g)$ be the hypersurface of rotation in
$\R^{n+1}$ given by $\{(x(s), y(s), z ) : s \in \R, z\in S^{n-1}\}$. 

The function $f$ on $M$ is a function of $s$ only, given by $f'(0)=0$ (the value of $f(0)$ is immaterial) and 
\begin{equation}\label{eq6.2}
    f''(s) =\begin{cases}
\kappa \left(1-\frac{D}{\pi r} \right), & s\in [0,\frac{\pi r}{2} -\delta],\\
\kappa \vp\left(\frac{s-\frac{\pi r}{2}}{\delta}\right) \left(1-\frac{D}{\pi r} \right) + \kappa\left(1-\vp\left(\frac{s-\frac{\pi r}{2}}{\delta}\right)\right), & s \in [\frac{\pi r}{2} -\delta, \frac{\pi r}{2} +\delta], \\
\kappa, & s\in [\frac{\pi r}{2} +\delta, D].
\end{cases}
\end{equation}
This implies $f'(D/2)=0$. We also extend $f$ to be even under the reflection in $s=0$ and $s=D/2$. 

With the above choices, let's compute the Bakry-\'Emery Ricci tensor of $(M,g,f)$. The eigenvalues of the second fundamental form are $k(s)$ (in the $s$ direction) and $\frac{\sqrt{1-(y')^2}}{y} $ in the orthogonal directions. 
Therefore the Gauss-Codazzi equations imply that the Ricci tensor has eigenvalues $(n-1)k(s)\frac{\sqrt{1-(y')^2}}{y} $ in the $s$ direction, and $k(s)\frac{\sqrt{1-(y')^2}}{y}  +(n-1)\frac{1-(y')^2}{y^2} $ in the orthogonal directions. 
The eigenvalues of the Hessian of $f$ can be calculated as follows: 
The curves of fixed $z$ in $M$ are
geodesics parametrized by $s$, so the Hessian in this direction is just $f''$ as given in \eqref{eq6.2}.
Since $f$ depends only on $s$ we also have that 
$\nabla^2 f(\p_s, e_i) =0$ for $e_i$ tangent to $S^{n-1}$ and $\nabla^2 f(e_i, e_j) =\frac{y'}{y}f' \delta_{ij}$.

The identities $y(s)=\int_0^s \cos(\theta(\tau)) d\tau$ and $y'(s)=\cos(\theta(s))$, where $\theta(s) =\int_0^s k(\tau)d \tau$ applied to \eqref{eq6.1} yields that as $\delta \to 0$, 
\begin{equation*}
    y(s)=\begin{cases} 
    r\sin(s/r), & s\in [0,\frac{\pi r}{2} -\delta ],\\
    r(1+o(\delta)), & s\in [\frac{\pi r}{2} -\delta ,D],
    \end{cases}
\end{equation*}
and 
\begin{equation*}
    y'(s)=\begin{cases} 
    \cos(s/r), & s\in [0,\frac{\pi r}{2} -\delta ],\\
    o(\delta), & s\in [\frac{\pi r}{2} -\delta , \frac{\pi r}{2} -\delta], \\
    0,  & s\in [\frac{\pi r}{2} +\delta , D].
    \end{cases}
\end{equation*}
Straightforward calculations then give the following expressions for the Bakry-\'Emery Ricci tensor $\Ric_f =\Ric +\nabla^2 f$: 
\begin{equation*}
    \Ric_f (\p_s, \p_s)=\begin{cases} 
    \kappa + \frac{n-1}{r^2} -\frac{\kappa D}{\pi r}, & s\in [0,\frac{\pi r}{2} -\delta ],\\
    \kappa + \vp\left(\frac{2-\frac{\pi r}{2}}{\delta} \right) \left(\frac{n-1}{r^2} (1+o(\delta)) -\frac{\kappa D}{\pi r} \right), & s\in [\frac{\pi r}{2} -\delta , \frac{\pi r}{2} -\delta], \\
    \kappa,  & s\in [\frac{\pi r}{2} +\delta , D],
    \end{cases}
\end{equation*}
and 
\begin{equation*}
    \Ric_f (e, e)=\begin{cases} 
    \frac{n-1}{r^2} +\frac{\kappa s \left(1-\frac{D}{\pi r}\right)}{r \tan(s/r)}, & s\in [0,\frac{\pi r}{2} -\delta ],\\
    \frac{n-2}{r^2} +o(\delta), & s\in [\frac{\pi r}{2} -\delta , \frac{\pi r}{2} -\delta], \\
    \frac{n-2}{r^2} \left(1+o(\delta) \right),  & s\in [\frac{\pi r}{2} +\delta , D],
    \end{cases}
\end{equation*}
while $\Ric_f(\p_s, e)=0$ for any unit vector $e$ tangent to $S^{n-1}$. 
Clearly, we have $\Ric_f \geq \kappa \, g$ for $r$ and $\delta$ small enough for any $\kappa \in \R$ if $n \geq 3$ and for $\kappa \leq 0$ if $n=2$. Moreover, the diameter of $M$ is $D(1+o(\delta))$. 

Finally, we show that for the Bakry-\'Emery manifold $(M, g, f)$ constructed above, the first nonzero eigenvalue $\l_{p,f}$ of $\Delta_{p,f}$ can be made arbitrarily close to $\mu_p(\kappa, D)$ by choosing $r$ and $\delta$ sufficiently small. 
Let $\psi$ be defined by 
\begin{equation*}
    \psi(z,s)= \begin{cases} w(s-D/2), & s\in [\frac{\pi r}{2}+\delta, D-\left(\frac{\pi r}{2}+\delta \right)],\\
    w(D/2 -\frac{\pi r}{2}-\delta), & s \in [0, \frac{\pi r}{2}+\delta] \cup [D-\frac{\pi r}{2}-\delta, D],
    \end{cases}
\end{equation*}
where $w$ is the solution of 
\begin{equation*}
    (p-1)|w'|^{p-2}w'' -\kappa s |w'|^{p-2}w' +\mu_p(\kappa, D-\pi r-2\delta) |w|^{p-2}w=0
\end{equation*}
with $w(0)=0$, $w'(0)=1$ and $w'(D/2-\frac{\pi r}{2} -\delta)=0$. 
Plugging the test function $\psi$ into the Rayleigh quotient gives 
\begin{eqnarray*}
    \l_{p,f} &\leq& \frac{\int_M |\nabla \psi |^p e^{-f} d\mu_g}{\int_M |\psi|^p e^{-f} d\mu_g}  \\
    &=& \frac{\mu_p(\kappa, D-\pi r-\delta)\int_{-\frac{\pi r}{2}-\delta/2}^{D-\frac{\pi r}{2}-\delta/2} |w|^p e^{-f} \, ds }{\int_0^D |w|^p e^{-f} \, ds } \\
    &\leq&  \mu_p(\kappa, D-\pi r-\delta)
\end{eqnarray*}
It follows that $\l_{p,f} \to \mu_p(\kappa, D)$ as $r$ and $\delta$ approach zero, thus proving the sharpness of the estimate \eqref{eigenvalue estimate} in Theorem \ref{Thm Main}.

\section*{Acknowledgments} {The authors would like to thank Professors Ben Andrews, Zhiqin Lu, Lei Ni, Guofang Wei and Richard Schoen for their interests in this work.}

\bibliographystyle{mrl}
\bibliography{ref}

\begin{thebibliography}{10}

\bibitem{Andrewssurvey15}
B.~Andrews, \emph{Moduli of continuity, isoperimetric profiles, and multi-point
  estimates in geometric heat equations}, in Surveys in differential geometry
  2014. {R}egularity and evolution of nonlinear equations, Vol.~19 of
  \emph{Surv. Differ. Geom.}, 1--47, Int. Press, Somerville, MA (2015).

\bibitem{AC11}
B.~Andrews and J.~Clutterbuck, \emph{Proof of the fundamental gap conjecture},
  J. Amer. Math. Soc. \textbf{24} (2011), no.~3,  899--916.

\bibitem{AC13}
---{}---{}---, \emph{Sharp modulus of continuity for parabolic equations on
  manifolds and lower bounds for the first eigenvalue}, Anal. PDE \textbf{6}
  (2013), no.~5,  1013--1024.

\bibitem{AN12}
B.~Andrews and L.~Ni, \emph{Eigenvalue comparison on {B}akry-{E}mery
  manifolds}, Comm. Partial Differential Equations \textbf{37} (2012), no.~11,
  2081--2092.

\bibitem{BQ00}
D.~Bakry and Z.~Qian, \emph{Some new results on eigenvectors via dimension,
  diameter, and {R}icci curvature}, Adv. Math. \textbf{155} (2000), no.~1,
  98--153.

\bibitem{DP05}
O.~Do\v{s}l\'{y} and P.~\v{R}eh\'{a}k, Half-linear differential equations, Vol.
  202 of \emph{North-Holland Mathematics Studies}, Elsevier Science B.V.,
  Amsterdam (2005), ISBN 0-444-52039-2.

\bibitem{Kroger92}
P.~Kr\"{o}ger, \emph{On the spectral gap for compact manifolds}, J.
  Differential Geom. \textbf{36} (1992), no.~2,  315--330.

\bibitem{Li79}
P.~Li, \emph{A lower bound for the first eigenvalue of the {L}aplacian on a
  compact manifold}, Indiana Univ. Math. J. \textbf{28} (1979), no.~6,
  1013--1019.

\bibitem{LY80}
P.~Li and S.~T. Yau, \emph{Estimates of eigenvalues of a compact {R}iemannian
  manifold}, in Geometry of the {L}aplace operator ({P}roc. {S}ympos. {P}ure
  {M}ath., {U}niv. {H}awaii, {H}onolulu, {H}awaii, 1979), Proc. Sympos. Pure
  Math., XXXVI, 205--239, Amer. Math. Soc., Providence, R.I. (1980).

\bibitem{LW19eigenvalue}
X.~Li and K.~Wang, \emph{Sharp lower bound for the first eigenvalue of the
  weighted $p$-Laplacian}, arXiv:1910.02295  (2019) \text{}.

\bibitem{Lott03}
J.~Lott, \emph{Some geometric properties of the {B}akry-\'{E}mery-{R}icci
  tensor}, Comment. Math. Helv. \textbf{78} (2003), no.~4,  865--883.

\bibitem{Matei00}
A.-M. Matei, \emph{First eigenvalue for the {$p$}-{L}aplace operator},
  Nonlinear Anal. \textbf{39} (2000), no. 8, Ser. A: Theory Methods,
  1051--1068.

\bibitem{NV14}
A.~Naber and D.~Valtorta, \emph{Sharp estimates on the first eigenvalue of the
  {$p$}-{L}aplacian with negative {R}icci lower bound}, Math. Z. \textbf{277}
  (2014), no. 3-4,  867--891.

\bibitem{Ni13}
L.~Ni, \emph{Estimates on the modulus of expansion for vector fields solving
  nonlinear equations}, J. Math. Pures Appl. (9) \textbf{99} (2013), no.~1,
  1--16.

\bibitem{SWW19}
S.~Seto, L.~Wang, and G.~Wei, \emph{Sharp fundamental gap estimate on convex
  domains of sphere}, J. Differential Geom. \textbf{112} (2019), no.~2,
  347--389.

\bibitem{Tolksdorf84}
P.~Tolksdorf, \emph{Regularity for a more general class of quasilinear elliptic
  equations}, J. Differential Equations \textbf{51} (1984), no.~1,  126--150.

\bibitem{Valtorta12}
D.~Valtorta, \emph{Sharp estimate on the first eigenvalue of the
  {$p$}-{L}aplacian}, Nonlinear Anal. \textbf{75} (2012), no.~13,  4974--4994.

\bibitem{WW09}
G.~Wei and W.~Wylie, \emph{Comparison geometry for the {B}akry-{E}mery {R}icci
  tensor}, J. Differential Geom. \textbf{83} (2009), no.~2,  377--405.

\bibitem{WZ17}
Y.~Zhang and K.~Wang, \emph{An alternative proof of lower bounds for the first
  eigenvalue on manifolds}, Math. Nachr. \textbf{290} (2017), no.~16,
  2708--2713.

\bibitem{ZY84}
J.~Q. Zhong and H.~C. Yang, \emph{On the estimate of the first eigenvalue of a
  compact {R}iemannian manifold}, Sci. Sinica Ser. A \textbf{27} (1984),
  no.~12,  1265--1273.

\end{thebibliography}

\end{document}